\documentclass[12pt, a4paper]{amsart}

\usepackage{amsmath, amssymb, amsthm, amsfonts}
\usepackage{mathrsfs}
\usepackage{mathtools}
\usepackage[top=2.6cm, bottom=2.5cm, left=2.6cm, right=2.6cm]{geometry}

\theoremstyle{plain}
\newtheorem{thm}{Theorem}
\newtheorem{lem}{Lemma}

\theoremstyle{remark}
\newtheorem{remark}{Remark}

\begin{document}

\title{$e$-invariants of quotients of Lie groups} 
\author{Haruo Minami}
\address{H. Minami: Professor Emeritus, Nara University of Education}
\email{hminami@camel.plala.or.jp}
\begin{abstract} 
Let $G$ be a simply connected compact simple Lie group and $\mathscr{L}$ be 
the left invarinat framing of $G$.  Let $\mathcal{L}^\lambda$ be the framing obtained by twisting $\mathscr{L}$ by a faithful representation $\lambda$. Given a torus subgroup $T''$ of $G$ we have a framing $(\mathcal{L}^\lambda)_{T''}$ of the quotient $G/T''$ induced from $\mathcal{L}^\lambda$. In this note we show that under a certain dimensional condition the $e_\mathbb{C}$-invariant of $G/T''$ with this framing provides a generator of the $J$-homomorphism or twice that. Thereby we also give a unified proof of the results for $SU(2n)$, $Spin(4n+1)$ and $Spin(8n-2)$ $(n\ge 1)$ previously proved.
\end{abstract}

\maketitle

\section{Introduction and main result}

Let $G$ be a simply connected compact simple Lie group of dimension $d$ and rank $m$ equipped with the left invariant framing $\mathscr{L}$. Then it is well known ~\cite{AS} that if $d\equiv -1\mod 4$ and $m\ge 2$, then
\[e_\mathbb{C}([G, \mathscr{L}])=0 \] where 
$e_\mathbb{C}\colon\pi_{4l-1}^S \to \mathbb{Q}/\mathbb{Z}$ is the complex 
$e$-invariant. In ~\cite{M, M2} we proposed to raise the problem of whether there exists a map $\lambda\colon G\to GL(t, \mathbb{R})$ such that the value of the $e_\mathbb{C}$-invariant of $[G, \mathscr{L}]$ with 
$\mathscr{L}$ replaced by the twisted framing $\mathscr{L}^\lambda$ by 
$\lambda$ gives a generator of the image of $e_\mathbb{C}$ in ~\cite{A}, and verified that this holds true for the cases $G=SU(2n)$, $G=Sp(4n+1)$ and $G=Spin(8n-2)$ $(n\ge 1)$. In this note we present a slightly more general result 
and thereby give a unified proof of these three results.

Let  $T$ be a maximal torus of $G$ and assume that it splits as a direct product $T=T'\times T''$ of its two subtori such that $\dim T'=2r-1$ $(r\ge 1)$ and so 
$\dim T''=m-2r+1$. By virtue of the Peter–Weyl theorem we have
a faithful representation of $G$ \[\rho : G\to GL(N, \mathbb{C}).\] 
Let $\lambda=k\rho_\mathbb{R}$ $(k\ge 0)$ be the direct sum of $k$ copies of the realification $\rho_\mathbb{R}$ of $\rho$. Let $(\mathscr{L}^\lambda)_{T''}$ denote the framing of $G/T''$ induced from the twisted framing 
$\mathscr{L}^\lambda$ by $\lambda$. Here $\rho$ is regarded as modified as follows: when a column of $\rho(g)$ $(g\in G)$ is acted on by an element of $\rho(T'')$, all its components $v$ including those belonging to other columns are converted into $|v|$ and in parallel the negatives of those $v$ are converted into $-|v|$. This allows to obtain $(\mathscr{L}^\lambda)_{T''}$ by applying the argument for a circle bundle of $(*)$ in  ~\cite{LS}  to the torus bundle 
$p\colon G\to G/T''$. (The explicit construction is given in Remark 2.)

We prove the following theorem. 

\begin{thm}  In the notation above, suppose $d-(m-(2r-1))=4l-1$, namely $G/T''$ has dimension $4l-1$. Then we have 
\begin{equation*}
e_\mathbb{C}\bigl([G/T'', (\mathscr{L}^{(r-1)\rho_\mathbb{R}})_{T''}]\bigr)
=(-1)^{l-1}B_l/2l
\end{equation*}
where $B_l$ denotes the $l$-th Bernoulli number.
\end{thm}

\begin{remark}
From the case $r=1$ above we obtain Example 4 in ~\cite{LS}, which  
in fact says that if $T''\subset T$ is a codimension 1 subtorus, then 
$e_\mathbb{C}\bigl([G/T'', \mathscr{L}_{T''}]\bigr)=(-1)^{l-1}B_l/2l$.
\end{remark}

The following theorem is the immediate consequence of Theorem 1.

\begin{thm} If $d=4l-1$ and $m=2r-1$, then we have 
\begin{equation*}
e_\mathbb{C}\bigl([G, \mathscr{L}^{(r-1)\rho_\mathbb{R}}]\bigr)
=(-1)^{l-1}B_l/2l.
\end{equation*}
\end{thm}

The proof of Theorem 1 is carried out based on the use of Proposition 2.1 of ~\cite{LS} along the same procedure as in the case $G=SU(2n)$. In fact it proceeds by intending to construct a tensor product decomposition of 
the complex line bundle $E$ associated to a certain circle bundle 
$S\to G/T''\to (G/T'')/S$, which enables us to apply the proposition above to 
$E$ and thereby leads us to the desired conclusion.

\begin{remark} By a similar argument to $(*)$ we have an analogous decomposition formula 
\begin{equation*} 
\underline{\mathbb{R}}^{m-2r+1}\oplus t(G)\cong p^*(t(G/T''))
\oplus\underline{\mathbb{C}}^{m-2r+1}
\end{equation*}
where $t(M)$ denotes the tangent bundle of $M$ and 
$\underline{\mathbb{K}}^s=G\times\mathbb{K}^s$ is the product bundle. 
Let $\underline{\mathbb{R}}^a\oplus t(G)\cong 
\underline{\mathbb{R}}^a\oplus (G\times t_e(G))$ $(a\ge 1)$ be the isomorphism induced by $\mathscr{L}^\lambda$ where $t_e(G)$ is the tangent space at the identity $e\in G$. Then we know that the right action of $T''$ on $t(G)$ operates  on the right-hand side as the adjoint action of $T''$, but since it is trivial in this case we have that it maintains the stable triviality of $t(G)/T''$. Hence by dividing the equation above by the right action of $T''$ we obtain  
\[\underline{\mathbb{R}}^{a+m-2r+1+d}\cong 
t(G/T'') \oplus\underline{\mathbb{R}}^{a+2m-4r+2}.\] 
This clearly provides a stable decompositon of $t(G/T'')$  
which we denote by $(\mathscr{L}^\lambda)_{T''}$. 
\end{remark}

\section{Tensor product decomposition of $E$}

We first review the case $G=SU(2)$ in order to introduce the notation for 
representing points of the 2-sphere. Put
\begin{equation*}
R(rz, u)=\begin{pmatrix} rz &u \\
-\bar{u} & r\bar{z} \end{pmatrix} \in SU(2)
\end{equation*}
where $r, \ s\ge 0, \ z\in S^1$, the unit sphere in $\mathbb{C}$, $u\in\mathbb{C}$ and let $d(z)=R(z, 0)$. Then $R(rz, u)d(z)=R(r, uz)$. Here if $r=0$, then since $|u|=1$, by replacing $z$ by $\bar{u}\bar{z}$, we have
$R(0, uz)d(\bar{u}\bar{z})=R(0, 1)$. For brevity we write 
\[(r, uz)_R=R(r, uz) \quad \text{with} \ uz=1 \ \text{when} \ r=0.\] 
Then the correspondence $(r, uz)_R\to (1-2r^2, 2rzu)$ yields a homeomorphism between $SU(2)/S^1$ and $S^2$ by viewing $S^1$ as the circle subgroup 
generated by $d(z)$ $(z\in S^1)$. Hence the principal bundle 
$S^1\to SU(2)\to SU(2)/S^1$ along with the projection map
$p\colon R(rz, v)\to (r, uz)_R$ becomes isomorphic to the Hopf bundle 
$S^1\to S^3\to S^2$. Below we identify $SU(2)/S^1=S^2$ and write ``$L$" for all the complex line bundles associated to the principal bundles constructed in such a way. 

Let $\mathfrak{g}$ and $\mathfrak{h}$ denote the complexifications of $\mathrm{Lie}(G)$ and  
$\mathrm{Lie}(T)$, the Lie algebras of $G$ and $T$, respectively.
Let $\alpha_1, \ldots, \alpha_n$ be the set of the positive roots of 
$\mathfrak{g}$ with respect to $\mathfrak{h}$. Then we know \!\!~\cite{H} that there exists a decomposion of $\mathfrak{g}$ such that
\begin{equation}
\mathfrak{g}=\mathfrak{h}\oplus\bigoplus_{i=1}^n
(\mathfrak{g}_{\alpha_i}\oplus\mathfrak{g}_{-\alpha_i}),\qquad \dim \mathfrak{g}_{\pm\alpha_i}=1
\end{equation}
and that for any $\alpha_i$ we have an element $H_{\alpha_i}\in \mathrm{Lie}(T)$ satisfying 
$\alpha_i(H_{\alpha_i})=2$ and 
\begin{equation}
\mathfrak{sl}(2, \mathbb{C})_{\alpha_i}=\mathfrak{g}_{\alpha_i}\oplus\mathbb{C}H_{\alpha_i}\oplus\mathfrak{g}_{-\alpha_i}
\end{equation}
which constitutes a subalgebra of $\mathfrak{g}$ isomorphic to 
$\mathfrak{sl}(2, \mathbb{C})=\mathfrak{su}(2)\otimes_\mathbb{R}\mathbb{C}$, 
denoted below by $\mathfrak{su}(2)_{\alpha_i}$. 

From now we regard $G$ as a closed subgroup of $GL(N, \mathbb{C})$ via $\rho$ 
and so consider as $\mathrm{Lie}(G)\subset \mathrm{Mat}(N, \mathbb{C})$ via $d\rho$. Let 
\begin{equation*}
R_i(r_iz_i, u_i):=
\left(
\begin{array}{ccccc} 
I_{a_i} & 0 & 0 & 0 & 0 \\ 
0 & r_iz_i & 0 & u_i & 0 \\ 
0 & 0 & I_{b_i} & 0 & 0 \\
0 & -\bar{u}_i& 0 & r_i\bar{z}_i & 0\\
0 & 0 & 0 & 0 & I_{c_i}
\end{array}
\right)\quad\text{with} \ \
R(r_iz_i, u_i)  \in SU(2)
\end{equation*}
in $GL(N, \mathbb{C})$ where $I_k$ is the identity matrix of size $k$. Then 
making a suitable choice of $\rho$ it becomes possible to assume that 
$\mathrm{exp}(tY_{\alpha_i})$ is written in the form
\[\mathrm{exp}(tY_{\alpha_i})=R_i(r_iz_i, u_i)\] 
for any $Y_{\alpha_i}\in\mathfrak{su}(2)_{\alpha_i}$, making a suitable choice of $\rho$,  we can write $\mathrm{exp}(tY_{\alpha_i})$ in the form
\[\mathrm{exp}(tY_{\alpha_i})=R_i(r_iz_i, u_i)\] 
where $\mathrm{exp}\colon \mathrm{Lie}(G)\to G$ is the exponential map. This 
also allows us to write 
\[\mathrm{exp}(tX_{\alpha_i})=R_i(x_i, 0)\qquad (x_i\in S^1)\] 
for $X_{\alpha_i}=iH_{\alpha_i}\in\mathfrak{su}(2)_{\alpha_i}$. 

For convenience we assume here that the order of $a_i$ and $b_i$ above is set as follows.
\begin{equation}
0\le a_1\le\cdots\le a_n \ \ \text{and} \   
\text{if} \ \ a_i=a_{i+1}, \ \text{then} \ \ b_i\le b_{i+1}.
\end{equation}

Following the above setting, put 
\[d(x)=R_1(x, 0)\cdots R_n(x, 0) \qquad (x\in S^1)\] and let
$S$ be the circle subgroup of $G$ generated by $d(x)$ $(x\in S^1)$.  
Let $S$ act on $G$ by the rule $(g, x)\to gd(\bar{x})$ and so if we let 
$[g] = gT''\in G/T''$, then $S$ acts on $G/T''$ by $d(x)[g] = [gd(\bar{x})]$. It is clear that the action of $S$ on $G/T''$ based on this rule is free due to the condition that $r\ge 1$. Hence putting $M=G/T''$ we can  regard $p\colon M\to M/S$ as the principal bundle along with the natural projection. Let $\pi\colon E=M\times_S\mathbb{C}\to M/S$ be the canonical line bundle over $M/S$ associated to $p$ where $S$ acts on $\mathbb{C}$ as $S^1$, that is,  $d(t)\cdot v=e^{ti}v$ for $v\in \mathbb{C}$. Then its unit sphere bundle 
$p'\colon S(E)\to M/S$ is  naturally isomorphic to $p$ as a principal $S$-bundle. 

For $1\le j\le n$, let us put
\[
R^{\{j\}}(r_iz_i, u_i)
=\textstyle\prod_{k=1}^nR_k(r_kz_k, u_k)\quad \text{with} \ r_k=1 \ \text{for all} \ k \ \text{except for} \ k=j
\]
and 
\begin{equation*}
\varGamma(r_iz_i, u_i)=R^{\{1\}}(r_iz_i, u_i)\cdots R^{\{n\}}(r_iz_i, u_i).
\end{equation*}
Then considering the setting of matrix size in (3) we have
\begin{lem}
$\mathrm{(i)}$ \ If $a_i\ne a_k$ and $b_i\ne b_k$, then $R^{\{j\}}(r_iz_i, u_i)$ and $R^{\{k\}}(x, 0)$ commute.\\
\,$\mathrm{(ii)}$ \ $R^{\{j\}}(r_iz_i, u_i)R^{\{j\}}(x, 0)=R^{\{j\}}(x, 0)
R^{\{j\}}(r_iz_i,, \bar{x}^2u_i)$. \\
$\mathrm{(iii)}$ \! $R^{\{j\}}(r_iz_i, u_i)R^{\{k\}}(x, 0)
=R^{\{k\}}(x, 0)R^{\{j\}}(r_iz_i, xu_i)$ in the cases $a_j=a_k+b_k$ and $a_k=a_j+b_j$, but otherwise   
$R^{\{j\}}(r_iz_i, u_i)$ and $R^{\{k\}}(x, 0)$ commute for any $j, k$.
\end{lem}
Let $d(x_1, \ldots, x_j)=d(x_1)\cdots d(x_j)$. Then 
from the lemma above we see that the action of $S$ on $R^{\{j\}}(r_iz_i, xu_i)$ and $R(r_iz_i, xu_i)$ can be given in the form
\begin{equation}
\begin{split}
R^{\{j\}}(r_iz_i, u_i)d(x_1, \ldots, x_{j-1})
&=d(x_1, \ldots, x_{j-1})R^{\{j\}}(r_iz_i, x_{\epsilon(j-1)}u_i), \\
\varGamma(r_iz_i, u_i)d(x_1, \ldots, x_n)
&=\textstyle\prod_{j=1}^nR^{\{j\}}(r_iz_i, x_{\epsilon(j-1)}u_i)d(x_j)
\end{split}
\end{equation}
for some 
$x_{\epsilon(k)}= x_k^{\epsilon_k}\cdots x_1^{\epsilon_1}\in S^1$ with 
$\epsilon_1, \ldots, \epsilon_k=-2, 0$\, or 1. Here note that in the above equations, if $r_i$ is zero, then since $x_{\epsilon(j-1)}u_i$ itself becomes an element of $S^1$ and so by replacing $x_i$ by 
$R^{\{j\}}(r_iz_i, x_{\epsilon(j-1)}u_i)$ on the right side can be converted to 
$R^{\{j\}}(0, 1)$.

Let $P^{\{j\}}$ $(1\le j\le m)$ be the subspace of $M$ consisting of 
$[R^{\{j\}}(r_iz_i, u_i)]$. Then we see that this subspace forms the total space of a principal $S$-bundle over $S^2$ along with the map $p^{\{j\}}\colon P^{\{j\}}\to S^2$ given by $[R^{\{j\}}(r_iz_i, u_i)]\to (r_i, \bar{z}_iu_i)_R$  which is clearly isomorphic to the complex Hopf bundle. In the notation above we write $L$ for the complex line bundle associated to $p^{\{j\}}$.
Let $(S^2)^k=S^2\times\cdots\times S^2$ ($k$ times). Then from (4), taking 
into account the argument after that, we see that there is a map 
$\phi_\varGamma\colon (S^2)^m \to M/S$ given by
\[
y=((r_1, w_1)_R, \ldots, (r_m, w_m)_R)\to p([\varGamma(r_i, w_i)]).
\] 
Let $P_\varGamma=\{[\varGamma(r_iz_i, u_i)]\mid R(r_iz_i, u_i)\in SU(2)\}\subset M$. Then 
looking at the equations of (4) we see that $P_\varGamma$ also forms the total space of a principal $S$-bundle endowed with the projection map 
$p_\varGamma\colon P_\varGamma\to (S^2)^n$ such that 
$\phi_\varGamma\circ p_\varGamma=p|P_\varGamma$. Let
$L^{\boxtimes k}$ denote the external tensor product $L\boxtimes\cdots\cdots\boxtimes L$ ($k$ times). Then we have 
\begin{equation}
\phi_\varGamma^*E\cong L^{\boxtimes{n}} 
\end{equation}
where $\phi_\varGamma^*E$ denotes the bundle induced by 
$\phi_\varGamma$.

Let $(S^2)^\circ$ be the subspace of $S^2$ consisting of $(r, w)_R$ with $r>0$ and let $((S^2)^k)^\circ$ be the direct product of $k$ copies of $(S^2)^\circ$. 
Then we have

\begin{lem}
The restriction of $\phi_\varGamma$ to $((S^2)^n)^\circ$ is an injective map.
\end{lem}
\begin{proof} Letting $y=(y_1, \ldots, y_n)$ with $y_k=(r_k, w_k)_R$, we suppose 
$\phi_\varGamma(y)=\phi_\varGamma(y')$, i.e.  
$p([\varGamma(r_i, w_i)])=p([\varGamma(r'_i, w'_i)])$ where we denote by attaching $``{ }_\prime$" to an element accompanied by $x$ its corresponding element by 
$x'$.  From the above we see that this can be interpreted as meaning that 
\begin{equation*}\tag{$\ast$}
R^{\{1\}}(r_i, w_i)\cdots R^{\{n\}}(r_i, w_i)
=R^{\{1\}}(r'_i, w'_i)\cdots R^{\{n\}}(r'_i, w'_i)
\end{equation*}
with $r_j>0$ on the left side and with $r'_j>0$ on the right side.
Here for brevity we suppose that, if necessary, by changing some $k$th rows and columns of the product matrices on the both sides of ($\ast$) with other rows and columns in the same order, the setting of $a_i$, $b_i$ in (3) is changed as follows. 
For some $1<s\le n$
\begin{equation*}
0=a_1=\cdots=a_s<a_{s+1}\le\cdots\le a_n, \ \ 
b_1=1, \ \ b_{j+1}=b_j+1 \ \ (j=1,\ldots, s-1).
\end{equation*}
Then the 1st column of the product 
$R^{\{1\}}(r_i, w_i)\cdots R^{\{s\}}(r_i, w_i)$ can be written as
\begin{equation*}\tag{$\ast\ast$}
\begin{split}
&a_{11}= r_1\cdots r_s, \ \  
a_{j1}= -r_{j+1}\cdots r_s\bar{w}_j \ \ (1\le j\le s-1), \\
&a_{s1}=-\bar{w}_s, \ \ a_{s+1\,1}=\cdots=a_{n1}=0.  
\end{split}
\end{equation*}

The above setting allows us to prove $y_1=y'_1, \ \ldots, y_s=y'_s$. 
First under the condition that $r_1\cdots r_s> 0$ we start by reverse induction using the equations for $a_{k1}$ of ($\ast\ast$). Applying the equation for 
$a_{s1}$ to ($\ast$)  we have $w_s=w'_s$, so it follows that $r_s=r'_s$ since 
$r_s^2+|w_s|^2={r'_s}^2+|w'_s|^2=1$ and $r_s>0$, $r'_s>0$. Next substituting this for the second equation for $a_{s-1\,1}$ we get $w_{s-1}=w'_{s-1}$, so we have 
$r_{s-1}=r'_{s-1}$ for the same reason above. Repeating this procedure we obtain $w_k=w'_k$ and $r_k=r'_k$ for $k=1, \ldots, s$ in reverse order.  Thus we have $y_k=y'_k$ for $k=1, \ldots, s$ 
under the assumption that $r_1\cdots r_s> 0$. Repeating this procedure again we finally arrived at the conclusion above. 

Simultaneously the result obtained implies that 
\[R^{\{1\}}(r_i, w_i)\cdots R^{\{s\}}(r_i, w_i)=R^{\{1\}}(r'_i, w'_i)\cdots 
R^{\{s\}}(r'_i, w'_i),\]
so by multiplying both sides of ($\ast$) by their inverses we have  
\[R^{\{s+1\}}(r_i, w_i)\cdots R^{\{n\}}(r_i, w_i)
=R^{\{s+1\}}(r'_i, w'_i)\cdots R^{\{n\}}(r'_i, w'_i).\]
In a similar way to the above case using this equation we obtain 
$y_k=y'_k$  for $k=s+1 \ldots, n$, which proves the lemma. 
\end{proof}

\section{Proof of Theorem 1}

Now by definition we can view $T'\subset M$ as a submanifold of dimension $2r-1$ and, for brevity, further suppose that $a_1, \ldots, a_{2r-1}$ and $b_1, \ldots, b_{2r-1}$ in (3) are set as follows.
\begin{equation*}\tag{$3'$}
a_1=0, \ a_{k+1}=a_k+1 \ (1\le k\le 2r-2) \quad \text{and} \quad 
b_k=0 \quad (1\le k\le 2r-1).
\end{equation*}
Following this, let $T'$ be generated by $R_1(z_1, 0)\cdots R_{2r-1}(z_{2r-1}, 0)$ 
$(z_i\in S^1)$ and put
 \[D^{\{k\}}(z_{2k-1}, xz_{2k})=R_{2k-1}(z_{2k-1}, 0)R_{2k}(xz_{2k}, 0)d(\bar{x})\qquad 
(1\le k\le r-1)\] 
where $x\in S^1$. 
Let $P_k\subset M$ be the subspace consisting of 
$[D^{\{k\}}(z_{2k-1}, xz_{2k})]$. Then it forms the total space 
of a principal $S$-bundle over $T^2=S^1\times S^1$ along with the projection map of $p_k\colon P_k\to T^2$ given by $[D^{\{k\}}(z_{2k-1}, xz_{2k})]\to 
(z_{2k-1}, xz_{2k})$ where $T^2$ is considered as an 
subspace of $M/S$ under $\iota_k\colon (z_{2k-1}, xz_{2k})\to 
p([D^{\{k\}}(z_{2k-1}, xz_{2k})])$. 

Let us put $z_{2k-1}=e^{\eta i}$, $z_{2k}=e^{\theta i}$ for  
$0\le \eta, \, \theta < 2\pi$ and let $\mu_k\colon T^2\to S^2$ be the map given by 
\begin{equation*}
(e^{\eta i}, xe^{\theta i}) \to \left\{
\begin{array}{ll}
(\cos(\eta/2),  xe^{\theta i}\sin(\eta/2))_R & \ (0\le \eta\le \pi)\\
(-\cos(\eta/2),  xe^{\theta t_\eta i}\sin(\eta/2))_R & \ (\pi\le \eta< 2\pi), \ \ \ 
t_\eta=2-\eta/\pi.
\end{array}
\right.
\end{equation*}
Then taking into account the fact that a principal circle bundle over $S^1$ is 
trivial we see that the classifying map of $p_k$ factors through $S^2$ where the restriction of $p_k$ to $\{1\}\times S^1\subset T^2$ is viewed as being trivial. 
Hence we have 

\begin{lem}[cf.\! ~\cite{LS}, \!\S2, Example 3] 
$p_k\colon P_k\to T^2$ is isomorphic to the induced bundle of the complex Hopf bundle $p\colon SU(2)\to S^2$ by $\mu_k$ and also 
$\mu_k^*\colon H^2(S^2, \mathbb{Z})\to H^2(T^2, \mathbb{Z})$ is an isomorphism for $1\le k\le r-1$.  
\end{lem}
\begin{proof}
In order to prove the first equation it suffices to show that there is 
a bundle map covering $\mu_l$. In fact we see that based on the above assumption  the assignment   
\begin{equation*}
D^{\{k\}}(e^{\eta i}, xe^{\theta i})\to\left\{
\begin{array}{ll}
R(\bar{x}\cos(\eta/2),  e^{\theta i}\sin(\eta/2)) & \ (0\le \eta\le \pi)\\
R(-\bar{x}\cos(\eta/2),  e^{\theta t_\eta i}\sin(\eta/2)) & \ (\pi\le\eta<2\pi), \ \ \ t_\eta=2-\eta/\pi
\end{array}
\right.
\end{equation*}
defines the desired bundle map $\tilde{\mu}_k\colon P_k\to SU(2)$. The second equation is immediate from the definition of $\mu_k$.
\end{proof}

Let us put 
\[D(z_{2i-1}, x_iz_{2i})=\textstyle\prod_{k=1}^{r-1}D^{\{k\}}(z_{2k-1}, x_kz_{2k})\qquad (x_k\in S^1).\] Then similarly to (4), based on Lemma 3 we know that we have
\begin{equation*}
\begin{split}
d(x)D^{\{k\}}(z_{2k-1}, \bar{x}^2x_kz_{2k})&=D^{\{k\}}(z_{2k-1}, x_kz_{2k})d(x)
\quad(1\le k\le r-1)\\
D(z_{2i-1}, x_iz_{2i})d(x'_1, \ldots, x'_{r-1})
&=\textstyle\prod_{k=1}^{r-1}D^{\{k\}}(z_{2k-1}, \bar{x}_{k-1}^2\cdots 
\bar{x}_1^2x_kz_{2k})d(x_k).
\end{split}
\end{equation*}
Taking into account these relations, if we let 
$\iota\colon (T^2)^{r-1}\to M/S$ be the map given by 
\[\lambda=(\lambda_1, \ldots, \lambda_{r-1})\to 
p([D(z_{2i-1}, x_iz_{2i})]), \quad\lambda_k=(z_{2k-1}, x_kz_{2k})\]
where $(T^2)^k=T^2\times\cdots\times T^2$ ($k$ times), then as in (5), by Lemma 3 we have
\begin{equation}
\iota^*E\cong L^{\boxtimes{(r-1)}}
\end{equation}
where $L$ denote the complex line bundles over $T^2$ associated to $p_k$ 
using the same symbol as the one over $S^2$ above.
Let $P_{\varGamma D}$ be the subspace of $M$ consisting of all the product 
elements $[\varGamma(r_{ij}z_{ij}, u_{ij})D(z_{2i-1}, x_iz_{2i})]$ and let 
$\phi_{\varGamma D}\colon (S^2)^n\times (T^2)^{r-1}\to M/S$ be the map given by 
\[(y, \lambda)\to p([\varGamma(r_i, w_i)D(z_{2i-1}, x_iz_{2i})])\]
where $y, \lambda$ are as above.
Then from the arguments for (5) and (6) we see that 
$P_{\varGamma D}$ forms the total space of a principal $S$-bundle endowed with the projection map $p_{\varGamma D}\colon P_{\varGamma D}\to (S^2)^n\times (T^2)^{r-1}$ such that 
$\phi_{\varGamma D}\circ p_{\varGamma D}=p|P_{\varGamma D}$. Due to this 
we also have 
\begin{equation}
\phi_{\varGamma D}^*E\cong L^{\boxtimes{(n+(r-1))}}.
\end{equation}
Here we write $(T^2)^\circ=T^2-\{1\}\times S^1$ and 
$((T^2)^k)^\circ$ for the direct product of $k$ copies of 
$(T^2)^\circ$. Then we have

\begin{lem}
The restriction of $\phi_{\varGamma D}$ to $((S^2)^n)^\circ\times 
((T^2)^{r-1})^\circ$ is an injective map.
\end{lem}
\begin{proof}
Suppose $\phi_{\varGamma D}(y, \lambda)=\phi_{\varGamma D}(y', \lambda')$ 
in terms of the notations used in the proof of Lemma 2.
Then by definition it can be viewed as the equation 
\begin{equation*}\tag{$\ast$}
\begin{split}
&\bigl(R^{\{1\}}(r_i, w_i)\cdots R^{\{m\}}(r_i, w_i)\bigr)\bigl(D^{\{1\}}(z_1, x_1z_2)\cdots D^{\{r-1\}}(z_{2r-3}, x_{r-1}z_{2r-2})\bigr)\\
&=\bigl(R^{\{1\}}(r'_i, w'_i)\cdots R^{\{m\}}(r'_i, w'_i)\bigr)\bigl(D^{\{1\}}(z'_1, x'_1z'_2)\cdots D^{\{r-1\}}(z'_{2r-3}, x'_{r-1}z'_{2r-2})\bigr)
\end{split}
\end{equation*}
where $r_k$ and $r'_k$ are $>0$. Following the setting of $a_i$, $b_i$ in (3) and 
($3'$) we know that when we present the product 
matrix on the left-hand side of ($\ast$) by $(a_{kl})$ the product of 
the first $r-1$ matrices $R^{\{1\}}(r_i, w_i), \ldots,  R^{\{r-1\}}(r_i, w_i)$  can be written as
\begin{equation*}\tag{$\ast\ast$}
\begin{split}
&a_{ss}=r_{s-1}r_s, \ \ (s=1, \ldots, 2r-1)\quad a_{s\, s-1}=-\bar{w}_{s-1} \ \ 
(s=2, \ldots, 2r-1), \\
&a_{s\, s+k}=
r_{s-1}r_{s+k}w_s\cdots w_{s+k-1} \ \ (s=1, \ldots, 2r-1, \ k=1, \ldots, 2r-2)\\ 
\end{split}
\end{equation*}
where $r_k=1$ when $k=0$, $k\ge 2r-1$. 
The proof is proceeded along the same lines as in Lemma 2.
We first want to prove the assertion $P(k)$ $(1\le k\le r-1)$ that
\begin{equation*}
\begin{split}
&(r_{2k-1}, w_{2k-1})=(r'_{2k-1}, w'_{2k-1}), \  (r_{2k}, w_{2k})=(r'_{2k1}, w'_{2k}), \  
z_{2k-1}=z'_{2k-1}, \\ &x_{2k-1}z_{2k}=x'_{2k-1}z'_{2k}; \ \text{i.e.} \ \ y_k=y'_k, \   
\lambda_k=\lambda'_k.
\end{split}
\end{equation*}
From ($\ast$) using ($\ast\ast$) we have
\begin{itemize}
\setlength{\leftskip}{-6mm}
\item[(i)] $r_1z_1=r'_1z'_1, \ \bar{w}_1z_1=\bar{w}'_1z'_1$.
\item[(ii)] $r_2w_1(x_1\bar{z}_1z_2)=r'_2w'_1(x'_1\bar{z}'_1z'_2), \ \ 
r_1r_2(x_1\bar{z}_1z_2)=r'_1r'_2(x'_1\bar{z}'_1z'_2), \ 
\bar{w}_2(x_1\bar{z}_1z_2)=\bar{w}'_2(x'_1\bar{z}'_1z'_2)$.
\end{itemize}
By (i) we see that since $r_1>0$,  $r_1^2={r'}_1^2$,  
so $r_1=r'_1$ and therefore it follows that $z_1=z_1'$. This shows that 
$P(1)$ holds true. In the same way applying these results to 
(ii) we have $(r_2, w_2)=(r'_2, w'_2)$ and $x_1z_2=x'_1z'_2$, namely that $P(2)$ holds true. More generall, from ($\ast$) and ($\ast\ast$) on the inductive hypothesis that $P(s)$ holds true in the previous step we know that 
\begin{itemize}
\setlength{\leftskip}{-5mm}
\item[($\mathrm{i}'$)] $r_{2s+1}w_1\cdots w_{2s}z_{2s+1}=r'_{2s+1}w'_1\cdots w'_{2s}z'_{2s+1},\\
r_lr_{2s+1}w_{l+1}\cdots w_{2s}z_{2s+1}=r'_lr'_{2s+1}w'_{l+1}\cdots w'_{2s}z'_{2s+1} \  (1\le l\le 2s-1),$\\  
$r_{2s}r_{2s+1}z_{2s+1}=r'_{2s}r'_{2s+1}z'_{2s+1}, \ 
\bar{w}_{2s+1}z_{2s+1}=\bar{w}'_{2s+1}z'_{2s+1}.$
\item[($\mathrm{ii}'$)]  $r_{2s+2}w_1\cdots 
w_{2s+1}(x_{2s+1}\bar{z}_{2s+1}z_{2s+2})=r'_{2s+2}w'_1\cdots 
w'_{2s+1}(x'_{2s+1}\bar{z}'_{2s+1}z'_{2s+2}),$ \\ 
$r_lr_{2s+2}w_{l+1}\cdots w_{2s+1}(x_{2s+1}\bar{z}_{2s+1}z_{2s+2})
=r'_lr'_{2s+2}w'_{l+1}\cdots w'_{2s+1}(x'_{2s+1}\bar{z}'_{2s+1}z'_{2s+2})$\\ 
$(1\le l\le 2s-1),$\\
$r_{2s}r_{2s+2}w_{2s+1}(x_{2s+1}\bar{z}_{2s+1}z_{2s+2})=r'_{2s}r'_{2s+2}w'_{2s+1}(x'_{2s+1}\bar{z}'_{2s+1}z'_{2s+2}),$
\\ 
$r_{2s+1}r_{2s+2}(x_{2s+1}\bar{z}_{2s+1}z_{2s+2})
=r'_{2s+1}r'_{2s+2}(x'_{2s+1}\bar{z}'_{2s+1}z'_{2s+2}),$\\ 
$\bar{w}_{2s+2}(x_{2s+1}\bar{z}_{2s+1}z_{2s+2})
=\bar{w}'_{2s+2}(x'_{2s+1}\bar{z}'_{2s+1}z'_{2s+2}).$ 
\end{itemize}
These equations allows us to repeat the inductive procedure. In fact,  
supposing that $P(k)$ holds true for $k=1, \ldots, s$ by use of ($\mathrm{i}'$), 
($\mathrm{ii}'$) we see that $P(s+1)$ holds true. This means that 
$y_k=y'_k, \lambda_k=\lambda'_k$ hold for $1\le k\le r-1$ and that 
\begin{equation*}
\begin{split}
&\bigl(R^{\{1\}}(r_i, w_i)\cdots R^{\{r-1\}}(r_i, w_i)\bigr)\bigl(D^{\{1\}}(z_1, x_1z_2)\cdots D^{\{r-1\}}(z_{2r-3}, x_{r-1}z_{2r-2})\bigr)\\
&=\bigl(R^{\{1\}}(r'_i, w'_i)\cdots R^{\{r-1\}}(r'_i, w'_i)\bigr)\bigl(D^{\{1\}}(z'_1, x'_1z'_2)\cdots D^{\{r-1\}}(z'_{2r-3}, x'_{r-1}z'_{2r-2})\bigr)
\end{split}
\end{equation*}
holds in $(\ast)$ and so the equation of $(\ast)$ can be rewritten as 
\begin{equation*}\tag{$\ast'$}
R^{\{r-1\}}(r_i, w_i)\cdots R^{\{m\}}(r_i, w_i)=R^{\{r-1\}}(r'_i, w'_i)\cdots 
R^{\{m\}}(r'_i, w'_i).
\end{equation*}
in the sense used above. Applying the procedure of the proof of Lemma 2 
to this $(\ast')$ we have $y_k=y'_k, \lambda_k=\lambda'_k$ $(r\le k\le n)$ 
and we can conclude that $(y, \lambda)=(y', \lambda')$. This proves the lemma.
\end{proof}

\begin{proof}[Proof of Theorem 1]
Put $W= (S^2)^n\times (T^2)^{r-1}$ and  
$W^\circ=((S^2)^n)^\circ\times ((T^2)^{r-1})^\circ$. From (1) we know that $d=m+2n$ and so $4l-2=d-(m-2r+1)=2n+2(r-1)$. Hence we have 
$\dim W=\dim M/S$. From this  and the injectivity result of $\phi_{\varGamma D}\,|\,W^\circ$ given in Lemma 4, taking into account the decomposition formula of (1), we know that 
$\phi_{\varGamma D}$ can be continuously deformed into an onto degree one map.
Let $[-]$ denote the fundamental class of a manifold $-$. Then  
we therefore have ${\phi_{\varGamma D}}_*([W])=[M/S]$, so 
\begin{equation}
\langle c_1(E)^{2l-1}, \, [M/S]\rangle
=\langle (c_1({\phi_{\varGamma D}}_*E)^{4l-1}), \, [W]\rangle. 
\end{equation}
Hence due to (7) we have 
\begin{equation*}
\begin{split}
\langle c_1(E)^{2l-1}, \, [M/S]\rangle
&=\langle c_1(L^{\boxtimes(n+(r-1))}, \, [W] \rangle\\
&=\langle  c_1(L^{\boxtimes n}), \, [(S^2)^n] \rangle\,
\langle c_1(L^{\boxtimes(r-1)}, \, [(T^2)^{r-1}]\rangle
\end{split}
\end{equation*}
Substituting this into the equation of Proposition 2.1 of ~\cite{LS} we obtain
\begin{equation}
e_{\mathbb C}([S(E), \Phi_E])=(-1)^{l-1}B_l/2l.
\end{equation}
Here $\Phi_E$ denotes originally the trivialization of the stable tangent space of 
$S(E)=M$ derived by the framing on $M/S$ induced by $\mathscr{L}$ 
due to it being $S$-equivariant (cf.\!~\cite{K}, p.\,42; ~\cite{LS}, p.\,36) 
when identifying $M/S=G/(T''\times S)$. But for the reason that 
this equation is obtained through the equation (8) we can see that the 
$\mathscr{L}$ above needs to replaced by $\mathscr{L}^{(r-1)\rho_\mathbb{R}}$. Therefore in this case $\Phi_E$ must coincide with 
$(\mathscr{L}^{(r-1)\rho_\mathbb{R}})_{T''}$. 

This replacement occurs because we have $r-1$ parts on $G$ on which 
$\mathscr{L}$ must be interpreted in two ways when observing the induced framing on $M/S$  via $\phi_{\varGamma D}$. This also is due to adding the terms 
$[D^{\{k\}}(z_{2k-1}, xz_{2k})]$ to that of the map $\phi_\varGamma$. Therefore 
it is essential to adopt the above replacement in order to remove this doubling. Consequently from (9) we obtain Theorem 1.
\end{proof}

\begin{remark}
As seen just above, if we apply the proof of Theorem 1 to that of Theorem 2, then the doubling of the framing occurred there can be dissolved by thinking of every 
$r-1$ line bundles $L'$s over the components $T^2$ of $(T^2)^{r-1}$ as a trivial complex line bundle.   But instead its first Chern class becomes zero and so, according to Proposition 2.1 of ~\cite{LS}, the value of $e_\mathbb{C}$ must become zero and therefore we see that $e_\mathbb{C}([G, \mathscr{L}])=0$ holds for any of the above $G$ with rank $\ge 3$, i.e. with $r\ge 2$ ~\cite{AS}.
\end{remark}

\end{document}